\documentclass[a4paper,12pt]{amsart}
\usepackage{enumitem} 
\usepackage[english]{babel}
\usepackage{amsthm,amsmath,amssymb}
\usepackage[dvips]{graphics,color}

\def\Ximp{X^{\mathrm{imp}}}
\def\Xage{X^{\mathrm{age}}}
\newtheorem{rema}{Remark}
\newtheorem{example}{Example}
\newtheorem{defi}{Definition}
\newtheorem{lemma}{Lemma}

\newtheorem{corollary}{Corollary}
\newtheorem{prop}{Proposition}
\newtheorem{thm}{Theorem}
\newtheorem{problem}{Problem}
\newtheorem{assumption}{Assumption}

\def \a   {\alpha}

\def \g   {\gamma}

\def \eps {\varepsilon}
\def\E{{\mathbb{E}}}
\def\F{\mathcal{F}}

\def\P{{\mathbb{P}}}
\def\R {{\mathbb{R}}}
\newcommand{\Z}{\mathbb{Z}}

\def\|{\,|\,}
\def\bn#1\en{\begin{align*}#1\end{align*}}
\def\bnn#1\enn{\begin{align}#1\end{align}}

\begin{document}
\title[Impatient RW]{Impatient random walk}
\date{\color{red}{\today}}
\author[J. Engl\"ander]{J\'anos Engl\"ander}
\address[J. Engl\"ander]{Department of Mathematics
\\ University of Colorado\\  Boulder, CO-80309-0395}
\email{janos.englander@colorado.edu.}
\urladdr{http://www.colorado.edu/math/janos-englander}
\author[S. Volkov]{Stanislav Volkov}
\address[S. Volkov]{Centre for Mathematical Sciences\\ Lund University\\ Lund 22100-118, Sweden}
\email{s.volkov@maths.lth.se}
\urladdr{http://www.maths.lth.se/~s.volkov/}
\thanks{The hospitality of Microsoft Research, Redmond and of the University of Washington, Seattle is gratefully acknowledged by the first author. Research of the second author was supported in part by the  Swedish Research Council grant VR2014--5157.}
\begin{abstract}
We introduce a new type of random walk where the definition of edge repellence/reinforcement is very different from the one in the ``traditional'' reinforced random walk models, and investigate its basic properties, such as null vs.\ positive recurrence, transience, as well as the speed. The two basic cases will be dubbed ``impatient'' and``ageing'' random walks.
\end{abstract}
\dedicatory{Dedicated to  B\'alint T\'oth on the occasion of his $60$th birthday}
\maketitle

\section{Introduction}\label{Intro}
\subsection{Model}
Consider an infinite connected  graph~$G$. The set of edges will be denoted by~$E(G)$. With some slight abuse of notation,  $v\in G$ will mean that~$v$ is a vertex of~$G$. Consider a random walk $X=\{X_n\}_{n\ge 0}$ on the vertex set of  $G$, with the jumps restricted to~$E(G)$. 

\begin{assumption}[Non-degeneracy]\label{Non-degeneracy}  The jumps have strictly positive probabilities for each edge in both directions. In particular, $X$ is an irreducible Markov chain on $G$.
\end{assumption}

Fix a vertex $v_0\in G$ which we call ``origin'', and assume that the walk starts at this point, $X_0=v_0$; for $G=\mathbb Z^d$, the default will be $v_0:=\mathbf{0}$.

\begin{defi}[Passage times]
A  sequence $s_0,s_1,s_2,\dots$ of nonnegative real numbers will be called a {\it sequence of passage times} if  $s_0=1$. 
\end{defi}
\begin{defi}[Walk modified by passage times]
We will modify the walk in such a way that if it has crossed an edge $e$ exactly  $k$ times before, then it takes  $s_k$ units of time (as opposed to $1$) to cross this edge again, in either direction; in particular, it takes  one unit of time to cross the edge for the first time. 
\end{defi}
The two basic cases are as follows:

\begin{defi}[Impatient and ageing walks]
Let $s_0,s_1,s_2,\dots$ be a given sequence of passage times. We will call the corresponding (modified) walk
\begin{description}
 \item[(i)] {\it impatient}\footnote{The intuitive meaning is clear:  the more the walker crosses the same edge, the faster it happens.}  when $s_k\downarrow 0$;
 \item[(ii)] {\it ageing} when $s_k\uparrow \infty$.
\end{description}
\end{defi}

\begin{rema}
If $s_k\downarrow s_\infty>0$ or $s_k\uparrow s_\infty<\infty$, then the questions (about recurrence, speed etc.) we investigate in this paper will be equivalent to the corresponding ones related to the original random walk, hence these cases are not interesting and are not considered in our paper.
\end{rema}

To have a more formal definition, note the main feature of the process we are studying: the system's ``actual'' time depends on the local time of the walk. 

\begin{defi}[Actual time]\label{actualtimedef}
Let
\begin{align*}
Z(e,m):=\sum_{i=1}^m {\mathbf 1}_{(X_{i-1},X_i)=e}, \quad e\in E(G),\ m\ge 1,
\end{align*}
be the number of occasions edge $e\in E(G)$ has  been crossed by time $m$. Then the \emph{actual time} after $m\ge 1$ steps is
\begin{align*}
T(m):=\sum_{k=1}^{m} s_{Z((X_{k-1},X_k),\, k-1)}.
\end{align*}
\end{defi}
In fact, it is more convenient to work in continuous time by extending the random function $T$ (time) as the nondecreasing function
\begin{align*}
T(t):=\sum_{k=1}^{\lfloor t\rfloor } s_{Z((X_{k-1},X_k),\, k-1)}
+(t-\lfloor t\rfloor )s_{Z((X_{\lfloor t\rfloor},X_{\lfloor t\rfloor+1}),\lfloor t\rfloor)},
\end{align*}
with the convention that for $t<1$, the value of the first sum is zero and then
$T(t)=t s_{Z((X_{0},X_{1}),0)}=t s_0=t.$

Let the random function $U:[0,\infty)\to[0,\infty)$ denote the right continuous generalized inverse of  $T$, that is, $U(t):=\sup\{s: T(s)\le t\}.$ With the exception of one case, we will work with $s_k>0$ for all $k\ge 0$, and then $T$ is strictly increasing in $t$ and $U=T^{-1}$. 

\begin{defi}[Definition of $\Ximp$ and  $X^{\mathrm{age}}$ via time change]\label{timechangedef}
The impatient random walk $\Ximp$  (ageing random walk $X^{\mathrm{age}}$) is defined via 
\begin{align*}
\Ximp(T(t))=X_{\lfloor t\rfloor},t\ge 0\ \ \ \ (X^{\mathrm{age}}(T(t))=X_{\lfloor t\rfloor},t\ge 0),
\end{align*}
or, equivalently, by
\begin{align*}
\Ximp(t):=X_{\lfloor U(t)\rfloor},t\ge 0\ \ \ (X^{\mathrm{age}}(t):=X_{\lfloor U(t)\rfloor},t\ge 0),
\end{align*}
where $U$ is as above.
Thus, $\Ximp$ and $X^{\mathrm{age}}$ move discontinuously\footnote{But one can also imagine that the walker is actually crossing the edge continuously with a speed depending on the passage time sequence.} according to the actual time.
\end{defi}

\subsection{Motivation}
Imagine that at every edge, one has to perform a certain task. For example, the edge represents a piece of road, where driving through is not trivial for some reason. Or that piece of connection between the  vertices is itself a small maze, one has to learn to solve. Then, the more one solved it in the past, the quicker it goes, and so the impatient walk models a learning process.\footnote{Our original model was more mundane: a person window shopping who gets bored quickly by the stores of any street she has already visited.}

Similarly, one can think of a model where the roads, or paths which are often used deteriorate with time, and therefore passing them becomes harder and harder and thus takes more and more time. It seems that ageing random walk can provide a good model for this situation.

While the model we introduce is somewhat reminiscent of the famous edge-reinforced random walk (see \cite{Pemantle} for a survey) as well as the ``cookie'' walk (introduced in~\cite{Zerner}), the behaviour in our model differs significantly from these latter ones, since in our case the transition probabilities remain intact while ``reinforcement'' affects only passage times.

\subsection{Notation}
As usual, $\Z_+$ will denote the set of non-negative integers and~$\Z^d$ will denote the~$d$-dimensional integer lattice. We will write~$a_n\asymp b_n$ if~$\lim_{n\to\infty} a_n/b_n=1$ and~$a_n\sim b_n$  if~$a_n/b_n=O(1)$ and~$b_n/a_n=O(1)$. The letters~$\P$ and~$\E$ will denote the probability and expectation corresponding to  the impatient/ageing random walk, respectively; when the starting point is emphasized, we will write $\P^{v_{0}}$ and $\E^{v_{0}}$.

\subsection{Questions we investigate}

One object we would like to study is $\widetilde\tau$,  the return time to the origin for the impatient walk, which is defined precisely as follows.

\begin{defi}[Actual return time $\widetilde\tau$]
Define $\tau_n$ and $\widetilde \tau_n$, $n=0,1,2,\dots$, by 
\begin{align*}
\tau_n(v_0)&:=\min\{k> \tau_{n-1}(v_0):\ X_{k}=v_0 \},\\
\widetilde\tau_n(v_0)&:=T(\tau_n(v_0)),
\end{align*} 
(with the implicit assumption $\tau_{-1}(v_0)=0$)
the latter being the total actual time spent by $\Ximp$ during the excursion from the origin to the origin starting at~$X^{\mathrm{imp}}(T(n))=v_0$.
\end{defi} 

An interesting phenomenon which arises in our model is that, depending on the passage times,~$\Ximp$  can be  positive recurrent even if the original walk was null-recurrent.

\begin{rema}
The distribution of $\widetilde\tau_n$ will in general depend on the history of the process  $\F_n:=\sigma\{X_0,X_1,\dots,X_{\tau_{n-1}}\}$.
\end{rema}

Another aspect of interest is the spatial speed (spread) of the process. We now need a definition.

\begin{defi}[Infinitely impatient walk]
Consider the walk on $G:=\Z_+$, and  an extreme case, when $s_k=0$ for all $k\ge 1$. That is, old edges are passed instantaneously. We call this walk the ``infinitely impatient walk'', and denote it by $X^{\mathrm{inf.imp}}$.
\end{defi} 

Clearly, $X^{\mathrm{inf.imp}}$ just steps to the right every time unit, because excursions to the left happen in ``infinitesimally small'' times. Hence, it spreads with constant speed. On the other hand, when $s_k=1$, for $k=1,2,...,$ we get the classical random walk for which the range up to $n$ scales with $\sqrt{n}$. So, this indicates that the scaling is always between $\sqrt{n}$ and $n$, and it depends on the passage times in some way. See also Remark \ref{strongweakspeed} and Theorem \ref{inf.imp.thm} below. This latter theorem will also shed some light on how the classical ArcSine Law is modified in our setting.

\subsection{Basic notions and a useful lemma}

When the passage times are summable, $\Ximp$ can only spend a finite amount of time (uniformly bounded by $S$) on any given edge. Accordingly, we make the following definition.
\begin{defi}[Strongly and weakly impatient random walks]\label{strong.weak}
The random walk will be called {\it strongly impatient} if
\\
(A1)\centerline{ $1\le S:=\sum_{k=0}^{\infty} s_k <\infty,$}
\\
and 
{\it weakly impatient} if
\\
(A2)\centerline{ $S:=\sum_{k=0}^{\infty} s_k =\infty.$}
\\
\end{defi}
\begin{rema}[Strong/weak impatience and speed]\label{strongweakspeed}
Clearly, when the walk is strongly impatient, it spends at most $S$ time on each edge, and so by time $t$ (we are talking about ``actual time'' here), it has visited at least $\lfloor t/S\rfloor$ different edges. This means that the spread (measured by the number of distinct edges crossed by the process up to time~$t>0$) is linear, with the constant being between~$1/S$ and~$1$ (because it has spent at least one unit of time on each edge visited). It would be desirable to figure out how the constant depends  on the passage times. A reasonable conjecture is that it is~$1/S$. See Section~\ref{speed.section}.

In order to actually change linearity and get closer to order~$\sqrt{t}$, one needs the walk to be weakly impatient. And so the question, in this case, is how the passage times will determine the order between linear and square root.
\end{rema}

Next, regarding recurrence, we make the following definitions.

\begin{defi}[Recurrence and positive recurrence]
We will call the impatient random walk
\begin{enumerate} 
\item[{\sf (B1)}]\emph{recurrent} if $\widetilde\tau_n(v_0)<\infty$, $\P^{v_{0}}$-a.s.  for all~$n\ge 0$ and all $v_0\in G$;
\item[{\sf (B2)}]\emph{transient} if it is not recurrent;
\item[{\sf (C1)}]\emph{positive recurrent} 
if~$\E^{v_{0}}  \widetilde\tau_n(v_0)<\infty$ \ for all~$n\ge 0$ and all $v_0\in G$;
\item[{\sf (C2)}]  \emph{null recurrent} if it is recurrent and $\E^{v_{0}}  \widetilde\tau_n(v_0)=\infty$ \ for all~$n\ge 0$ and all $v_0\in G$.
\end{enumerate}
\end{defi}

Clearly, if $\widetilde\tau_n(v_0)<\infty$ a.s.  for {\it some}~$n\ge 0,v_0\in G$, then $X$ must be recurrent, in which case~$\widetilde\tau_n(v_0)<\infty$ a.s.\ for {\it all}~$n\ge 0$ and $v_0\in G$.

A similar statement holds for positive recurrence. \begin{thm}[Process property]\label{proc.prop} Recall Assumption~\ref{Non-degeneracy} and assume also that $X$ on $G$  is recurrent. Then  $X^{\mathrm{imp}}$  is either positive recurrent or null recurrent. In other words,  the   properties in {\sf (C1--C2)} do not depend on the choice of~$n$ or~$v_0$. 
\end{thm}
\begin{proof} We prove this statement by verifying that:
\begin{enumerate}
\item[{\sf (i)}] for given $v_0\in G$, the property does not depend on $n\ge 0$;
\item [{\sf (ii)}] for $n=0$, the property does not depend on choice of $v_0\in G$.
\end{enumerate}
{\sf (i)} Assume first that $\E^{v_{0}}  \widetilde\tau_0(v_0)<\infty$. Then, because of the monotonicity of $s$, $\E^{v_{0}}  \widetilde\tau_n(v_0)<\infty$  for~$n>0$, as well.

Now assume that $\E^{v_{0}}  \widetilde\tau_0(v_0)=\infty$. Let $e=(v_0,v_1)$ be an outgoing edge  from $v_0$. If $n\ge 1$  then (because of the non-degeneracy of $X$ on $G$) the event
$$
A_n:=\{X_0=v_0,X_1=v_1,X_2=v_0,X_3=v_1,\dots, X_{2n-1}=v_1,X_{2n}=v_0\}
$$
(i.e.\  the first $n$ excursions consist only of traversing $e$ back and forth, so that this edge has been crossed $2n$ times) has a positive probability.

On~$A_n$, the  next passage time on $e$ has been set to $s_{2n}$, while no other edge has been crossed. It is enough to show that 
\begin{equation}\label{needs.pf}
\E (\widetilde\tau_n(v_0)\mid A_n)=\infty.
\end{equation}

Now, \eqref{needs.pf} would certainly be true without the impatience mechanism of the model, as we assumed that  $\E^{v_{0}} \widetilde\tau_0=\infty$, so our task is to prove that this mechanism does not change the validity of \eqref{needs.pf}.

To this end, let $\eta$ denote the number of crossings of $e$ in the $n+1$st excursion starting at~$v_0$. Then, either $\eta=0$ (when both the initial and the last edges are different from $e$), or $\eta=1$ (when the final edge is $e$ and the initial edge is not, or vice versa), or $\eta=2$ (when both the initial and the last edges are~$e$). Recall the notion of ``actual time'' from Definition~\ref{actualtimedef}. If $p_j:=\P(\eta=j)$ for~$j=0,1,2$, then the expected actual time spent on $e$ in the excursion with $s_0=1$ initial passage time is $p_1 +p_2 (1+s_1)$, whereas with $s_{2n}$ initial passage time it is $p_1 s_{2n}+p_2(s_{2n}+s_{2n+1}).$ These are finite quantities and therefore resetting the initial time to $s_{2n}$ from $s_0$ does not change the finiteness of the expected actual return time. 

\smallskip
{\sf (ii)}  Assume now that $a:=\E^{v_0} \widetilde\tau_0(v_0)<\infty$. We will show that for any $v_1\neq v_0$ we also have~$\E^{v_1} \widetilde\tau_0(v_1)<\infty$. Since $G$ is connected, without the loss of generality, we may (and will) assume that they $v_1$ and $v_0$ are neighbours on $G$.

Let $p_{0\to 1}:=\P^{v_{0}}(X_1=v_1)>0$.
Note that $\E(\kappa\| X_0=v_0,X_1=v_1)<\infty,$ where 
\begin{align*}
\kappa&:=T(\min\{k\ge 1:\  X_k=v_0\}),\quad \text{since }\\
a&\ge p_{0\to 1}\cdot \E(\kappa\| X_0=v_0,X_1=v_1).
\end{align*}
Once we have this, using an argument very similar to the one in part (i), one can see that $\E^{v_1} (\kappa)<\infty$ also holds.

Let $X_0=v_1$.  
The argument below shows even that the expected time it takes to return to $v_1$ {\it after visiting $v_0$ en route} is finite. Note that, because of the monotonicity of $s$, the expectation of the actual time for $\Ximp$ to return to $v_0$ is not more than $a$; the same is true for the consecutive excursions from $v_0$ to $v_0$. However, during each of these excursions the walk $X$ will visit $v_1$ with probability at least $p_{0\to 1}$. Therefore, if $\xi$ is an auxiliary  geometric random variable with range $\{0,1,2,...\}$ and parameter $p_{0\to1}$ under the law $\mathbf P$, then, using that $\E^{v_{1}} \kappa<\infty$, we have
$$
\E^{v_1}\widetilde\tau_0(v_1)\le \E^{v_1} \kappa + 1+a\mathbf {E}\xi=\E^{v_1} \kappa + 1+ \frac{a(1-p_{0\to1})}{p_{0\to1}}<\infty,
$$
completing the proof. 
\end{proof}

\medskip
Given that the original walk is recurrent, so are the impatient and ageing random walks,  since  $T(n)<\infty$, whenever $n<\infty$. The question of positive vs.\ null-recurrence is not so trivial though. The statement below follows from the obvious inequality~$\widetilde\tau\le \tau$ for $\Ximp$.

\begin{rema}
If the original walk $X$ is positive recurrent  then $\Ximp$ is also positive recurrent.
\end{rema}

Finally, define
$M:={\sf card}\{(X_{i-1},X_i),\ i=1,\dots,\tau_1(v_0)\}$, that is, $M$ is the number of distinct edges crossed by $X$ between consecutive visits to the origin.  The following statement involving~$M$ will be useful later.

\begin{lemma}[Size of excursion area]\label{lemmacr} The impatient walk is
\begin{enumerate}
\item[(i)] positive recurrent, provided that the walk is strongly impatient and $\E M<\infty$;
\item[(ii)] null-recurrent, provided that $\E M=\infty$. 
\end{enumerate}
\end{lemma}

\begin{proof}
Let $v_0\in G$.
In the strongly impatient case, $$M\le \tau_1(v_0)\le SM,$$
and so $\E_{v_{0}} M\le \tau_1(v_0)\le S\E_{v_{0}} M.$ The lower estimate $M\le \tau_1(v_0)$ is, however,  always true, whether the impatience is strong or not. Hence the result follows from these observations and Theorem~\ref{proc.prop}.
\end{proof}

\begin{corollary}[Switching from null to positive. recurrence]
If the original walk $X$ is null-recurrent but $\E M<\infty$ holds, then strong impatience turns null-recurrence into positive recurrence (for~$\Ximp$).
\end{corollary}

This is the case, for example, for the random walk mentioned in the second part of  Theorem~\ref{thmphase} later.

\section{Impatient simple random walk on $\Z^d$, $d=1,2$, and generalization}\label{Imp.SRW.d12}
It turns out that for $d=1,2$, impatience cannot change the null-recurrent character of the simple random walk.

\begin{thm}\label{for.any.sequence}
The impatient simple random walks on $\Z^1$ and on $\Z^2$ are null recurrent for any sequence of passage times.
\end{thm}

\begin{proof} Let $M$ be as before.

\underline{$d=1$:} We may assume that $v_0=\mathbf{0}$ and that the first step of the walk is to the right, $X_1=1$. The probability to reach vertex $m\ge 1$ before returning to the origin is  $1/m$, hence $\P(M\ge m)=1/m$, $m=1,2,\dots$. (Note that the number of distinct edges visited by the one-dimensional walk coincides with its maximum during the excursion). Consequently $\E M=\infty$ and by Lemma~\ref{lemmacr}  the impatient random walk is null-recurrent.

\underline{$d=2$:} one can easily show (e.g.\ by coupling) that the quantity $M$  for $d=2$ is stochastically larger than for $d=1$, so $\E M=\infty$ here and we again can use Lemma~\ref{lemmacr}.
\end{proof}

\begin{rema}
As the following example shows, one can easily construct a null recurrent random walk on $\Z^2$ (and similarly on $\Z^d$, $d\ge 3$) such that the strongly impatient random walk with the same transition probabilities has $\E M<\infty$ and hence is positive recurrent.
\end{rema}

Indeed, consider the following random walk on $\Z^2$.  Let $|(x,y)|:= \max(|x|,|y|)$ be the ``norm" of a vertex, let orbit $O_k=\{(x,y):  |(x,y)|=k\}$, $k=1,2,3,...$
that is, the squares with sides of length $2k$ parallel to the axes and the centre in the origin (please see the picture, each orbit has a distinctive colour).

Assume that the transition probabilities are the following: once the walker is on $O_k$ it goes to the previous orbit $O_{k-1}$ with probability  $\frac 23 \cdot 2^{-k}$,  or to the next orbit $O_{k+1}$ with probability $\frac 13\cdot 2^{-k}$, and with the remaining probability $1-2^{-k}$ it goes clock-wise staying on $O_k$ (this needs to be adjusted somewhat at the four corners of $O_k$).

If we consider the embedded process, defined as the norm of the point whenever it changes, then it is a positive recurrent random walk with probability $2/3$ going towards the origin, and $1/3$ going away from the origin.
Also, when the walker reaches orbit $O_k$, it stays there a geometric number of steps, on average $2^k$ steps.

The probability to reach orbit $O_k$ before returning to the origin is of order $2^{-k}$. Therefore, the average number of steps the walk makes is of order
$
\sum_{k=1}^{\infty} 2^{-k} \cdot 2^k=\infty.
$
At the same time, the average number of distinct vertices the walker visits before returning to the origin is bounded by 
$
\sum_{k=1}^{\infty} 2^{-k} \cdot 8k < \infty
$
since each orbit has $8k$ distinct vertices.
Consequently, the expected number of distinct vertices (and edges) visited by the walk is finite.

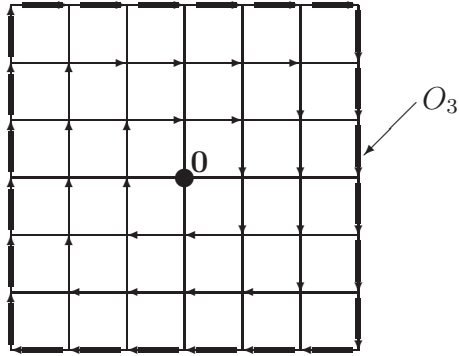
\setlength{\unitlength}{0.03in}
\begin{picture}(60,65)(-40,0)
\put(71,31){\bf 0} 
\put(70,30){\circle*{3}} 
\multiput(40,0)(0,10){7}{\line(1,0){60,0}}
\multiput(40,0)(10,0){7}{\line(0,1){60,0}}

\multiput(50,00)(10,0){6}{\vector(-1,0){10,0}}
\multiput(60,10)(10,0){4}{\vector(-1,0){10,0}}
\multiput(70,20)(10,0){2}{\vector(-1,0){10,0}}
\multiput(40,60)(10,0){6}{\vector(1,0){10,0}}
\multiput(50,50)(10,0){4}{\vector(1,0){10,0}}
\multiput(60,40)(10,0){2}{\vector(1,0){10,0}}
\multiput(100, 0)(0,10){6}{\vector(0,-1){0,10}}
\multiput(90, 10)(0,10){4}{\vector(0,-1){0,10}}
\multiput(80, 20)(0,10){2}{\vector(0,-1){0,10}}

\multiput(40,10 )(0,10){6}{\vector(0,1){0,10}}
\multiput(50,20)(0,10){4}{\vector(0,1){0,10}}
\multiput(60,30)(0,10){2}{\vector(0,1){0,10}}

\linethickness{0.025in}
\multiput( 40, 1)(0,10){6}{\line(0,1){7,0}}
\multiput(100, 2)(0,10){6}{\line(0,1){7,0}}
\multiput( 42, 0)(10,0){6}{\line(1,0){7,0}}
\multiput( 42,60)(10,0){6}{\line(1,0){7,0}}

\put(111,42){$O_3$}\put(110,43){\vector(-1,-1){9}}
\end{picture}
\\[0mm]

A generalization of the one-dimensional case is as follows. Let $X$ be a nearest-neighbour walk with the outward drift $b(x)\in (-1,1)$ at vertex $x\in \Z^1$, that is,
\begin{align*}
&b(0)=0;\\
&b(x):=\P(X_{n+1}=x+1\| X_n=x)-\P(X_{n+1}=x-1\| X_n=x),\ \text{for}\ x>0;\\
&b(x):=\P(X_{n+1}=x-1\| X_n=x)-\P(X_{n+1}=x+1\| X_n=x),\ \text{for}\ x<0;
\end{align*}
and for $m\ge 1$, define
\begin{align*}
B^r(m)&:=\sum_{x=1}^m \prod_{k=1}^x \frac{1-b(k)}{1+b(k)},\quad B^l(m):=\sum_{x=1}^m \prod_{k=1}^x \frac{1-b(-k)}{1+b(-k)},
\\
I^r(b)&:=\sum_{x=1}^{\infty} \frac{1}{1+B^r(x)},\quad I^l(b):=\sum_{x=1}^{\infty} \frac{1}{1+B^l(x)},
\end{align*}
so that $I^r(b),I^l(b)\in [0,\infty].$
\begin{thm}[Size of drift]\label{inward.drift.thm}
Assume that $X$ on $\Z^1$ is null recurrent.
\begin{enumerate}
\item[(a)]
If either $I^r(b)=\infty$ or $I^l(b)=\infty$ then $X^{\mathrm{imp}}$ on $\Z^1$ is null recurrent as well, for any sequence of passage times.
\item[(b)]
If  $I^r(b),I^l(b)<\infty$ then $X^{\mathrm{imp}}$ on $\Z^1$ is positive recurrent  whenever the walk is strongly impatient.
\end{enumerate}
\end{thm} 
\begin{proof} 
Clearly, positive recurrence holds exactly when the expected return time is finite whether we condition it on $X_1=1$ or on $X_1=-1$. Assuming, for example, that $X_1=1$, we now link the finiteness of $I^r(b)$ with the finiteness of the expected return time; an analogous argument holds for the case when $X_1=-1$, concerning the finiteness of $I^l(b)$, so we omit it. Thus suppose $X_1=1$ from now on. By Lemma~\ref{lemmacr}, it is enough to show that 
\begin{align}\label{EM.I}
\E M=1+I^r(b),
\end{align}
in the sense that both sides are either finite and equal, or they are both infinite.
 
To compute $\E M$ for the nearest-neighbour one-dimensional random walk on $\Z$, note that now $M$ is the rightmost  lattice point reached in an excursion, and use the electrical network representation (see e.g.~\cite{RWEN}) with resistors $R_x$ located between $x$ and $x+1$ satisfying
$$
\frac{R_x}{R_{x-1}}=\frac{\P(X_{n+1}=x-1\| X_n=x)}{\P(X_{n+1}=x+1\| X_n=x)}
=\frac{\frac12-\frac{b(x)}2}{\frac12+\frac{b(x)}2}=\frac{1-b(x)}{1+b(x)}.
$$
Without the loss of generality, we may set $R_0:=1$. Therefore
$$
R_x=\prod_{k=1}^x \frac{1-b(k)}{1+b(k)};\quad B^r(x)=\sum_{y=1}^x R_y.
$$
Consequently,
$$
\P (M> x)=\P (x+1\ \mathrm{reached})=\frac{1}{1+R_1+\dots+R_{x}}=\frac{1}{1+B^r(x)},\ x\ge 1,
$$ yielding~\eqref{EM.I}.
\end{proof}

\section{Lamperti and Lamperti-type walks}
In Section~\ref{Imp.SRW.d12} we have seen that for the symmetric random walk  in dimensions one and two, impatience cannot turn null recurrence into positive recurrence. We are, therefore, going to consider certain random walks which are ``just barely null recurrent'',   and show that in those cases impatience can actually make them positive recurrent. Our models will be related to the ``Lamperti-walk'' (see~\cite{lamp}).

We first need a general result, presented in the next subsection.

\subsection{A general formula for nearest-neighbour  walk on $\Z_+$ }\label{general.formula}

Let $G:=\Z_+$ and consider a nearest-neighbour ageing or impatient random walk $X^*$ (i.e. $X^*=\Ximp$ or $X^*=\Xage$), with the passage times $\{s_i\}$, where the underlying random walk $X$ is positive recurrent.
For $m\ge 2$, let (with using $\P$ for $X$ too)
\begin{align*}
p_m&:=\P^1(X\text{ reaches $m$ before reaching }0),
\\
q_m&:=\P^m(X\text{ reaches $m+1$ before reaching }0).
\end{align*}
\begin{lemma}\label{eqmajor.lemma}
For $X^*$, the expected actual length (in time) of the first excursion from $0$, denoted by $\widetilde \tau=\widetilde \tau_1(0)$  is
\begin{align}\label{eqmajor}
\E_0\widetilde\tau=1+s_1+\sum_{m=2}^{\infty} p_m
\sum_{j=0}^\infty  q_m^{j+1} \left(s_{2j}+s_{2j+1}\right).
\end{align}
\end{lemma}
\begin{proof} Since every time the walk traverses the edge $(m,m+1)$ rightwards, it must traverse it again before reaching $m$ again, it follows that
\begin{align*}
\E_0\widetilde\tau&=1+s_1+\sum_{m=2}^{\infty} p_m \E (\mathrm{time\ spent\ traversing \ }(m,m+1)\| X_0=m)
\\&=1+s_1+\sum_{m=2}^{\infty} p_m
\sum_{j=0}^\infty  q_m^{j+1} \cdot\text{time spent traversing back and forth for}\ (j+1)^{st}\ \text{time}
\\&=1+s_1+\sum_{m=2}^{\infty} p_m
\sum_{j=0}^\infty  q_m^{j+1} \left(s_{2j}+s_{2j+1}\right),
\end{align*}
as claimed.
\end{proof}
In fact, using the  method of electric networks just like in the proof of Theorem~\ref{inward.drift.thm} (see again~\cite{RWEN}), we are in possession of the useful formulae:
\begin{align}
p_m&=\frac{1}{1+R_1+\dots+R_{m-1}},\label{useful.form.p}
\\
q_m&=\frac{1+R_1+\dots+R_{m-1}}{1+R_1+\dots+R_{m-1}+R_m},\label{useful.form.q}
\end{align}
where $R_i, i\ge 1$ are the resistors of the electrical network corresponding to our random walk.

\subsection{Lamperti walk: $b(x)\asymp c/x$  on $\Z_+$}

As an application to Lemma~\ref{eqmajor.lemma}, we obtain a theorem concerning a case when short enough passage times can turn null-recurrence  into positive recurrence. But first we need the following statement (see e.g.\ Proposition~7.1 (i)--(iii) in~\cite{SHU}).
\begin{prop}
Let $X$ be a random walk  on $\Z_+$, with drift  $b(x)\asymp c/x$. Then $X$ is positive recurrent if $c<-1/2$,
null-recurrent for $c\in [-1/2,1/2]$, and transient for $c>1/2$.
\end{prop}

\begin{thm}\label{thmphase}
Let $X$ be a recurrent random walk  on $\Z_+$, with drift  $b(x)\asymp c/x$ (recurrence means~$c\le 1/2$). Furthermore, let the passage times satisfy $s_j\asymp j^{-\a}$, $\a>0$. Then  $\Ximp$ is positive recurrent if and only if $c<\min\left\{0,\frac{\a-1}{2}\right\}$.

In particular, when $c\in [-1/2,0)$, $X$  is null recurrent but $\Ximp$ is positive recurrent whenever $\a>1+2c$ (for example, when $\a>1$, i.e. the impatience is strong).
\end{thm}
\begin{rema}
In fact, $\Ximp$ is positive recurrent for any strongly impatient walk when $c<0$. (The proof is similar.)
\end{rema}

The diagram below summarizes the results in Theorem \ref{thmphase}.

\vspace{0.5cm}
\setlength{\unitlength}{1cm}
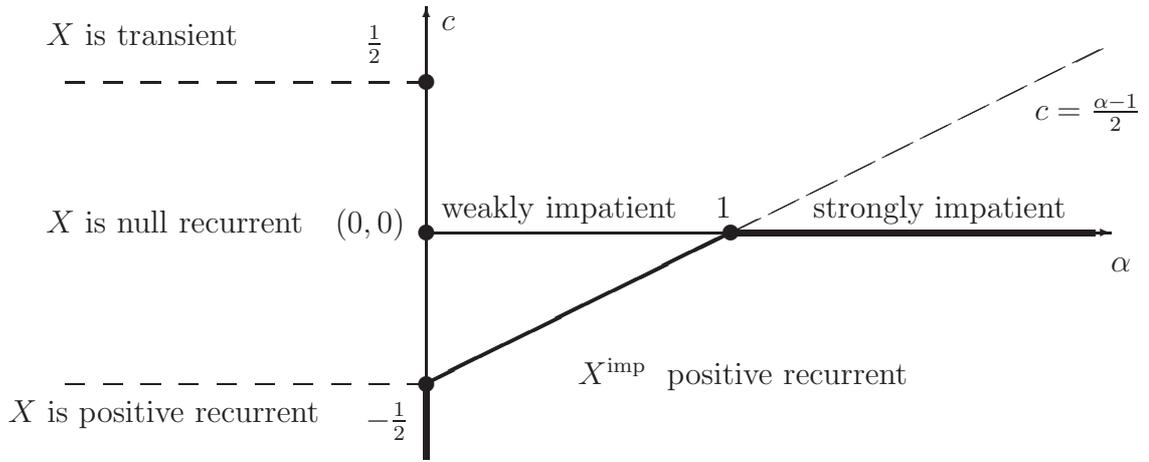
\begin{picture}(13,6)
\put(4.8,3){$(0,0)$} 

\put(6,3){\circle*{0.2}}
\put(10,3){\circle*{0.2}}
\put(6,1){\circle*{0.2}}
\put(6,5){\circle*{0.2}}

\put(9.8,3.2){$1$}
\put(5.2,5.4){$\frac{1}{2}$}
\put(5.2,0.4){$-\frac{1}{2}$}

\put(6,3){\vector(1,0){9}}
\put(6,0){\vector(0,1){6}}
\put(15,2.5){$\a$}
\put(6.2,5.7){$c$}

\put(6,1){\line(2,1){4}}
\multiput(10,3)(0.5,0.25){10}{\line(2,1){0.4}}

\multiput(6.02,1)(-0.01,-0.00){9}{\line(2,1){3.95}}
\multiput(5.97,1)(0.01,0){7}{\line(0,-1){1}}
\multiput(10,2.97)(0,0.01){7}{\line(1,0){4.8}}

\put(14,4.5){$c=\frac{\a-1}{2}$}

\multiput(6,5)(-0.5,0){10}{\line(-1,0){0.25}}
\multiput(6,1)(-0.5,0){10}{\line(-1,0){0.25}}

\put(0.5,0.5){$X$ is positive recurrent}
\put(1,3){$X$ is null recurrent}
\put(1,5.5){$X$ is transient}

\put(6.2,3.2){weakly impatient \quad \quad \quad \quad strongly impatient}
\put(8,1){$\Ximp$ \ positive recurrent}

\end{picture}

\vspace{0.5cm}
\begin{proof}[Proof of Theorem~\ref{thmphase}.]
As before,  the resistors satisfy $$R_m=\prod_{k=1}^m \frac{1-b(k)}{1+b(k)},\ m\ge 1.$$ Using the Taylor approximation of $\log(1+z)$ for small $|z|$ along with the integral approximation of monotone series, it is an easy exercise to show that, defining $R_0:=1$, as $m\to\infty$,
$$
R_m\sim m^{-2c}\ \text{and}\ \sum_{i=0}^m R_i\sim m^{1-2c},
$$ 
provided $c\neq \frac{1}{2}$. (When, $c=\frac{1}{2}$, one has $\sum_{i=0}^m R_i\sim \log m$;  this case  will be treated at the end of the proof.) Now~\eqref{useful.form.p} implies that, as $m\to\infty$,
$$
p_m=\frac{1}{\sum_{i=0}^{m-1} R_i}\sim \frac{1}{m^{1-2c}}.
$$
Similarly, using \eqref{useful.form.q},
$$
a_m:=1-q_m=\frac{R_{m}}{\sum_{i=0}^m R_i}\sim  \frac{1}{m}.
$$
By~\eqref{eqmajor} and the monotonicity of $s_j$, in order to establish whether $\Ximp$ is positive recurrent or not, we thus have to analyse the finiteness of 
$$
\sum_{m=0}^{\infty} \frac{\mathsf{SUM}(m)}{m^{1-2c}},
\quad\text{where}\quad
\mathsf{SUM}(m):=\sum_{j=1}^{\infty} \left(1-a_{m}\right)^{j+1} \frac{1}{j^{\a}}.
$$ 
Clearly, if $\mathrm{const}$ denotes a  constant in $(0,m)$, then
$$
\sum_{j=1}^{\infty} \left(1-\frac{\mathrm{const}}{m}\right)^{j+1} \frac{1}{j^{\a}}\ge 
\sum_{j=1}^{m} \left(1-\frac{\mathrm{const}}{m}\right)^{j+1} \frac{1}{j^{\a}}\ge
\sum_{j=1}^{m} \left(1-\frac{\mathrm{const}}{m}\right)^{m+1} \frac{1}{j^{\a}},
$$
and since
$(1-\mathrm{const}/m)^{m+1}\uparrow e^{-\mathrm{const}}$ as $m\to\infty$, one has
$$
\sum_{j=1}^{m} \left(1-\frac{\mathrm{const}}{m}\right)^{m+1} \frac{1}{j^{\a}}
\sim   \sum_{j=1}^m \frac{1}{j^{\a}}\sim
h_{\a}(m):=
\begin{cases}
 m^{1-\a}, &\text{ if } \a<1;\\
 \log m, &\text{ if } \a=1;\\
 1, &\text{ if } \a>1.
\end{cases}
$$
Since $c_1/m\le a_m\le c_2/m$ for some $0<c_1<c_2$, it follows that 
$$
\mathsf{SUM}(m)\ge c_3 h_{\a}(m),
$$
with some $c_3>0$.

We now show that on the other hand 
\begin{align}\label{sumleftbound}
\mathsf{SUM}(m)\le c_4 h_{\a}(m),
\end{align}
holds with some other constant $c_4>0$ (possibly depending on $\a$). It will then follow that 
$$
\E\widetilde{\tau}\sim \sum_m \frac{h_\a(m)}{m^{1-2c}}=
\begin{cases}
\sum_m m^{2c-\a} & \text{ if } \a<1;\\
\sum_m \frac{\log m}{m^{1-2c}} & \text{ if } \a=1;\\
\sum_m m^{2c-1} & \text{ if } \a>1,
\end{cases}
$$
proving the statement of the theorem. To verify \eqref{sumleftbound}, denote  $\g:=e^{-c_1}\in (0,1)$; one has then
\begin{itemize}[leftmargin=*]
\item  if $\a>1$ then $\mathsf{SUM}(m)\le \sum_{j=1}^{\infty} j^{-\a}<\infty$; 
\item  if $\a=1$ then using Riemann-sum approximation for the function $f(x)=1/x$,
\begin{align*}
 \mathsf{SUM}(m)&\le 
\sum_{j=1}^{m} \frac{1}{j}
+
\sum_{j=m+1}^{\infty} \left(1-\frac{c_1}{m}\right)^{j+1} \frac1j
\le
\sum_{j=1}^{m} \frac 1j
+
\sum_{k=1}^{\infty} \g^{k}  \sum_{j=km+1}^{(k+1)m} \frac 1j\\ 
&\le \log m+\sum_{k=1}^{\infty} \g^{k}  \int_{km}^{(k+1)m} \frac 1j= \log m + \sum_{k=1}^{\infty} \log(1+1/k)\g^{k} \sim \log m = h_1(m);
\end{align*}

\item  if $\a<1$ then 
\begin{align*}
\mathsf{SUM}(m)&\le 
\sum_{j=1}^{\infty} \left(1-\frac{c_1}{m}\right)^{j+1} \frac{1}{j^{\a}}
\le 
\sum_{k=0}^{\infty} \g^{k}  \sum_{j=km+1}^{(k+1)m} \frac{1}{j^{\a}}
\le 
\sum_{k=0}^{\infty} \g^k  \frac{m}{(km+1)^{\a}}\\
&\le m^{1-\a} \sum_{k=0}^{\infty} \frac{\g^k}{k^\a}  \sim m^{1-\a}=h_{\a}(m),
\end{align*}
\end{itemize}
and thus~\eqref{sumleftbound} has been established.

\medskip
It remains to consider the case $c=\frac{1}{2}$. Now we have $p_m\sim \frac{1}{\log m}$ and~$\mathsf{SUM}(m)\ge O(1)$, so~$\E \tilde\tau \sim \sum_{m} p_m \mathsf{SUM}(m)\ge \sum_m \frac{1}{\log m}=\infty$ for any value of $\a\ge 0$, ruling out positive recurrence for~$\Ximp$. Since $X$ is recurrent, so is $\Ximp$, and thus $\Ximp$ is in fact null recurrent in this case.
\end{proof}

\subsection{Lamperti-type walk on $\Z_{+}$ with very small inward drift}
Consider a random walk on the non-negative integers, and assume now that the drift is much weaker than for the original Lamperti case, that is, that for large $x$,
\begin{align*}
b(x)\asymp\frac{D}{ x \log x},
\end{align*}
where $D<0$. 
While the corresponding walk $X$ is obviously null-recurrent for any $D\in\R$ (compare it with the classical Lamperti case), the following result shows that  there is a phase transition for the behaviour of  the impatient walk $\Ximp$.
\begin{thm} $\Ximp$ is
\begin{enumerate}
\item [(i)] null-recurrent for any $D\ge -1/2$ and any sequence of passage times;
\item[(ii)] positive-recurrent for any $D<-1/2$, provided it is strongly impatient.
\end{enumerate}
\end{thm}
\begin{proof}
We use Theorem~\ref{inward.drift.thm}. We have
$$
\prod_{k=1}^m \frac{1-b(k)}{1+b(k)}\sim 
\left(\log m\right)^{-2D} \quad \Longrightarrow \quad B^r(x)\sim x \left(\log x\right)^{-2D},
$$
hence $I^r(b)<\infty$ if and only if $D<-1/2$.
\end{proof}

\section{Expectation calculations  for hitting times and positive recurrence in one-dimension; the passage generating function}
In this section we perform some calculations concerning one-dimensional hitting times, and analyse one-sided positive/null recurrence for the walk.

We start with notation. We assume $d=1$, and $\sigma_n, n\in \Z$ will denote the hitting time by $X^{\mathrm{imp}}$ of $n$. Then $T_n=\sigma_{-n}\wedge \sigma_n$ is the exit time from $(-n,n)$ (i.e. the hitting time of $\{-n,n\}$), when starting at $-n\le x\le n$.
Introduce the shorthand $s_j^*:=s_{2j-2}+s_{2j-1}$ for $j\ge 1,$ and note that $\sum s_j^*=\sum s_j\in [1,\infty]$.
Let $u\in \mathbb Z^1$ and $v=u+h,\ h\ge 1$. With $n$, $u$ and $h$ (and thus $v$ too) fixed, define the following quantities:
\begin{enumerate}
\item For two-sided hitting times, define $$\rho_m^{(x)}:=\mathbb P_{x}(X\ \text{reaches}\ m\ \text{before}\ \pm n)=\mathbb P_{x}(T_n>\tau_{m}),$$ for $-n\le x,m\le n$.

\item For one-sided hitting times define $$r_m^{(x)}:=\mathbb P_{u+x}(X\ \text{reaches}\ u+m\ \text{before}\ v)=\mathbb P_{u+x}(\tau_v>\tau_{u+m}),$$ for all $x,m\in \mathbb Z^1$ and note that
$r_m^{(x)}=0$ for $m\ge h$ and $r_m^m=1$. 
\end{enumerate}

\begin{rema}
In a concrete situation, the quantities $\rho_m^{(x)},r_m^{(x)}$ are of course, computable as  ratios involving resistors. 
\end{rema}

\subsection{Passage generating function and passage radius}
The following notion will be useful  for impatient as well as ageing walks.

\begin{defi}
[Passage generating function and passage radius] 
The power series
$$
\phi(z):=\sum_{j=1}^{\infty} s_j^*z^j
$$
will be called the \emph{passage generating function}. In particular, $\phi(1)=\sum s_j^*=\sum s_j\in [1,\infty]$. The corresponding radius of convergence will be called the~\emph{passage radius} for $\phi$:
$$
R^{\mathrm{pass}}:=
\frac{1}{\limsup_{k}(s_k^*)^{1/k}}.
$$
\end{defi}

Of course, for the original motion $X$ ($s_j\equiv1$), one has $$\phi^{\mathrm{orig}}(z)=\frac{2z}{1-z},\ \ R^{\mathrm{pass}}=1.$$
\begin{rema}
[Strong impatience, super-ageing and $R^{\mathrm{pass}}$] Note the following.
\begin{itemize}
\item[(i)] It is clear that strong impatience implies that $R^{\mathrm{pass}}\ge 1$, while weak impatience implies that $R^{\mathrm{pass}}= 1.$ 

\item[(ii)] In the strongly impatient case, we can normalize the passage times so that $\sum s_j^*=\sum s_j=1$ (at the expense of speeding up time by a constant factor), and then $\phi$ is actually a probability generating function. This has the practical advantage that we can use well-known formulae for $\phi$ in the strongly impatient case.

\item[(iii)] In the ageing case, it is possible that $R^{\mathrm{pass}}=0$ (``super-ageing").
\end{itemize}
\end{rema}

\subsection{``Positive recurrence to the right'' (PRR) for impatient walk}

Let $X$ be a recurrent walk on $\mathbb Z^1$. 
\begin{defi}[One sided positive recurrence]
We say that  $X^{\mathrm{imp}}$ is  {\it positive recurrent to the right} if $\E_u \sigma_v<\infty$ for $u<v$ and {\it null recurrent  to the right} if $\E_u \sigma_v=\infty$ for $u<v$. We will show below (see Remark~\ref{easy.to.see}) that this definition does not actually depend on choice of~$u$ or~$v$.

The definition for the positive/null recurrence to the left (when $u>v$) is analogous.
\end{defi}

Our fundamental result about one-sided positive recurrence is as follows.

\begin{thm}[Criterion for PRR]\label{PRR.crit}
$X^{\mathrm{imp}}$ is positive recurrent to the right if and only if
\begin{align}\label{eq:crit}
    \sum_{m=-\infty}^{0} r_m^{(0)}\phi\left(r_{m-1}^{(m)}\right)<\infty.
\end{align}
\end{thm}

\begin{proof}
Let  $u<v\in\mathbb Z^1$ and define

$(I):=\E_u$(actual time spent between $u$ and $v$ until hitting $v$);

$(II):=\E_u$(actual time spent between $-\infty$ and $u$ before hitting $v$).

If $v-u=h$, then we can calculate $(I)$ by considering the expected local actual times on the edges between $u$ and $v$ as follows.
\begin{align*}
(I)
&=\sum_{m=1}^{h} \E_u (\text{actual time spent traversing the edge}\ (u+m-1,u+m))=s_{{0}}h\\
&+\sum_{m=1}^{h} \sum_{j=1}^{\infty} 
\text{(act.\ time spent trav.\ 
for $j^{th}$ occasion after $1^{st}$ upcrossing})
\cdot \P(\exists j^{th}\ \text{occasion})
\\ &
=s_{{0}}h+\sum _{m=1}^{h} \left( \sum _{j=1}^{\infty } \left( r^{(m)}_{m-1} \right) ^{j} \left( s_{{2\,j-1}}+s_{{2\,j}} \right)  \right).
\end{align*} 
In the impatient regime,  $s_j\le C$ for $j\ge 1$ with some $C>0$, and  as $r^{(m)}_{m-1}<1$ for all $m$, the last term is bounded by
$$
2C\sum _{m=1}^{h} \sum_{j=1}^{\infty}\left(r^{(m)}_{m-1}\right)^j=2C\sum _{m=1}^{h} \frac{r^{(m)}_{m-1}}{1-r^{(m)}_{m-1}}.
$$ 
Consequently, $(I)<\infty$.

Similarly, we can calculate $(II)$ by considering the expected local actual times on the edges which are ``below'' $u$.

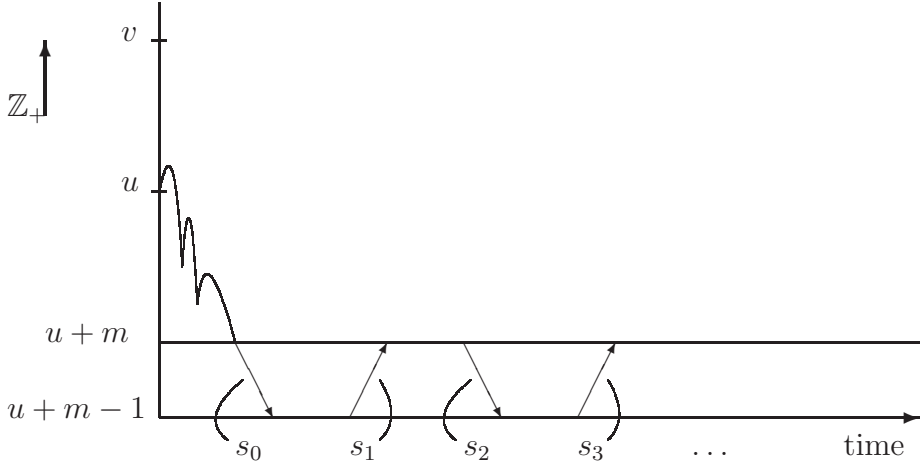
\setlength{\unitlength}{1cm}
\begin{picture}(12,7)
\thicklines
\put(0.5,5){\vector(0,1){1}}
\put(0,5){$\mathbb Z_{+}$}
\put(11,0.5){time}
\put(0,1){$u+m-1$}
\put(0.5,2){$u+m$}
\put(1.5,6){$v$}\put(1.9,6){\line(1,0){0.2}}
\put(1.5,4){$u$}\put(1.9,4){\line(1,0){0.2}}
\put(2,1){\vector(1,0){10}}
\put(2,2){\line(1,0){10}}
\put(2,1){\line(0,1){5.5}}

\put(3,   0.5){$s_0$}
\put(4.5, 0.5){$s_1$}
\put(6, 0.5){$s_2$}
\put(7.5, 0.5){$s_3$}
\put(9, 0.5){$\dots$}

\thinlines
\put(3,2){\vector(1,-2){0.5}}
\put(6,2){\vector(1,-2){0.5}}
\put(4.5,1){\vector(1,2){0.5}}
\put(7.5,1){\vector(1,2){0.5}}

\qbezier(2,4)(2.2,5)(2.3,3)\qbezier(2.3,3)(2.4,4.5)(2.5,2.5)\qbezier(2.5,2.5)(2.6,3.5)(3,2)

\qbezier(2.9, 0.7)(2.5, 1)(3.1, 1.5)
\qbezier(5.9, 0.7)(5.5, 1)(6.1, 1.5)
\qbezier(4.9, 0.7)(5.2, 1)(4.9, 1.5)
\qbezier(7.9, 0.7)(8.2, 1)(7.9, 1.5)
\end{picture}

Hence,
\begin{align*}
    (II)
    &=\sum_{m=-\infty}^{0} \E_u (\text{actual time spent traversing the edge}\ (u+m-1,u+m))\\ 
    &=\sum_{m=-\infty}^{0} r_m^{(0)}\sum_{j=1}^{\infty} \text{(actual time spent traversing left and right for}\ j^{th}\ \text{occasion)}\\
    &\ \ \ \ \cdot \P(\exists j^{th}\ \text{ocassion})\\
    &=\sum_{m=-\infty}^{0} r_m^{(0)}\sum_{j=1}^{\infty}\left[r_{m-1}^{(m)}\right]^js_j^*=
    \sum_{m=-\infty}^{0} r_m^{(0)}\phi\left(r_{m-1}^{(m)}\right).
\end{align*}

Clearly, $\E_u \sigma_v =(I)+(II),$ and, since $(I)<\infty$, we are done.
\end{proof}

\begin{rema}[Consistency of the definition]\label{easy.to.see}
It is easy to see that the condition~\eqref{eq:crit} does not depend on the choice of $u$ or $v$ (for non-degenerate~$X$). For example, for a fixed  $u$, if we increase $h$, as this is shown in the above proof: only $(II)$ will change; $(I)$ is always finite.

A similar calculation shows that the condition is invariant under fixing $v$ and changing $u$.
\end{rema}

We can refine Theorem \ref{PRR.crit} as follows.

\begin{corollary}
Consider the following, simpler condition, involving the original walk only:
\begin{align}\label{simpler.cond}
 \sum_{m=-\infty}^{0} r_m^{(0)}<\infty.
\end{align}
\noindent (a)  Then \eqref{simpler.cond} is  a necessary condition for the PRR property for
$X^{\mathrm{imp}}$, no matter what the passage times are, as long as we rule out that $\lim_{m\to-\infty}r_{m-1}^{(m)}=0$.

\noindent (b) If either the impatience is strong, or  $r:=\sup_{m\in\mathbb{Z}}r_{m-1}^{(m)}<1$ then \eqref{simpler.cond} is  a sufficient condition for PRR for
$X^{\mathrm{imp}}$.
\end{corollary}

\begin{proof}

\noindent (a) This follows from Theorem \ref{PRR.crit} and the fact that $\phi(t)$ is bounded away from zero if $t>\epsilon>0$.

\noindent (b) For strong impatience, $\phi\left(r_{m-1}^{(m)}\right)\le \phi(1)<\infty$ for all $m$, giving the assertion. If $r<1$, then $\phi\left(r_{m-1}^{(m)}\right)\le \phi(r)<\infty$ for all $m$, and we are done again.
\end{proof}

\begin{rema}[Original process] 
Taking $s_k=1, k\ge 1$, we obtain 
$$\sum_{m=-\infty}^{0}\frac{2r_m^{(0)}
r_{m-1}^{(m)}}{1-r_{m-1}^{(m)}}<\infty
$$
as the criterion for the positive recurrence of $X$, in which case the positive recurrence of $X^{\mathrm{imp}}$ follows immediately. So, to avoid this trivial situation, we can always assume that
$$
\sum_{m=-\infty}^{0}\frac{2r_m^{(0)}r_{m-1}^{(m)}}{1-r_{m-1}^{(m)}}=\infty
$$
\end{rema}

\begin{example}[SRW]
For simple random walk, no passage time sequence can lead to positive recurrence to the right.\footnote{We have already verified this with another method -- see Theorem \ref{for.any.sequence}.} Indeed, for $m<0$ we have $r_m^{(0)}=\frac{h}{h-m}$ and $r_{m-1}^{(m)}=\frac{h-m}{h-m+1}$. Thus,  for every $j\ge 1$,
$$
\sum_{m=-\infty}^{0} r_m^{(0)}[r_{m-1}^{(m)}]^j=\sum_{m=-\infty}^{0} h\frac{(h-m)^{j-1}}{(h-m+1)^j}=\infty,
$$
and so the quantity in~\eqref{eq:crit} after change in the order of summation equals
$$
\sum_j s_j^*\sum_{m=-\infty}^{0} r_m^{(0)}\left[r_{m-1}^{(m)}\right]^j=\infty.
$$
\end{example}
More generally, we have
\begin{corollary} 
If there exists a $j\ge 1$ for which
$$
\sum_{m=-\infty}^{0} r_m^{(0)}\left[r_{m-1}^{(m)}\right]^j=\infty,
$$
then $X^{\mathrm{imp}}$ is null recurrent to the right.
\end{corollary}

\subsection{Expected exit times -- two sided} 

In this section we consider  ageing random walks.\footnote{It is easy to see that in the impatient case, the expectation we study is always finite for any irreducible walk.} The first piece of  information about the speed we are aiming to obtain is the expected actual time to reach~$\pm n$ starting from the origin, that is, $\E_0 T_n$. Let $n\ge 2$ and $0\le m\le n-2$. Each edge $(m,m+1)$ can be crossed $0,1,2,\dots$ times before the walk reaches $\pm n$, unlike the cases for recurrence to the right, where the  parity of those times is fixed. Note also that the actual time spent on the edge $(n-1,n)$ is always either $0$ or $s_0$ and hence finite, as the walk can traverse it at most once before exiting $(-n,n)$. Similar statement holds for the edge $(-n, -n+1)$.

Once the walk started at $0$ reaches $m$, which happens with probability $\rho_m^{(0)}$, it can either  exit~$(-n,n)$ without ever crossing $(m,m+1)$ (in fact, it must be then $-n$), or cross this edge at least once -- the latter happens with probability $\rho_{m+1}^{(m)}$. If the walk reached $m+1$, a similar argument shows that  to cross $(m,m+1)$ once again, going leftwards, before exiting $(-n,n)$ has probability $\rho_{m}^{(m+1)}$. Consequently, the expected  actual time spent on the edge $(m,m+1)$ before exiting the interval equals 
\begin{align*}
{\ }&\rho_m^{(0)} \left[
\rho_{m+1}^{(m)} \left(s_0+\rho_{m}^{(m+1)}\left(s_1
+\rho_{m+1}^{(m)} \left(s_2+\rho_{m}^{(m+1)}\left(s_3
+\dots \right)\right)\dots \right)
\right)\right]
\\ &=
\rho_{m+1}^{(0)}\left[s_0+\gamma_m s_2 +\gamma_m^2 s_4+\gamma_m^3 s_6+\dots\right]
+
\rho_{m}^{(0)}\left[\gamma_m s_1+\gamma_m^2 s_3 +\gamma_m^3 s_5+\dots\right],
\end{align*}
where we defined
\begin{align}\label{eqgam}
\gamma_m:=\rho_{m+1}^{(m)} \rho_{m}^{(m+1)},
\end{align}
and  used that, by the Markov property, $\rho_m^{(0)} \rho_{m+1}^{(m)}=\rho_{m+1}^{(0)}$. Hence, if~$\widetilde {T}_n^{+}$ denotes the total actual time spent on the edges  $(0,1),(1,2)...(n-2,n-1)$ before exiting the interval, then 
\begin{align*}
\E_0 \widetilde {T}_n^{+}&=\sum_{m=0}^{n-2}
\left[ \rho_{m+1}^{(0)}\sum_{k=0}^\infty \gamma_m^k s_{2k}
+ \rho_{m}^{(0)} \sum_{k=1}^\infty \gamma_m^k s_{2k-1} \right].
\end{align*}
By the irreducibility of the walk,  $0<\rho^{(0)}_m<1$ and $0<\gamma_m<1$ for all relevant $m$, and so  we conclude that $\E_0 \widetilde {T}_n^{+}$ is finite if and only if
$$
\phi(\gamma_m)<\infty \quad \text{for all}\quad m=0,1,\dots,n-2.
$$
Similarly, let  $\widetilde {T}_n^{-}$ denote the  total actual time spent on the edges  $(-n+1,-n+2),...,(-1,0)$  before exiting the interval. Conducting an analogous argument,  one can compute $\E_0 \widetilde {T}_n^{-}$, which, taking into account that
$$
0\le T_n- \left({\widetilde {T}}_n^{-}+\widetilde {T}_n^{+}\right)\le 2 s_0,
$$
leads to the following criterion.

Recall Assumption~\ref{Non-degeneracy}.
\begin{thm}
Consider a one dimensional walk  and let $n\ge 2$. Then $\E_0  \widetilde {T}_n<\infty$ if and only if
$$
\max_{m=-(n-1),\dots,n-2}\phi(\gamma_m)<\infty
$$
where $\gamma_m$ is defined by~\eqref{eqgam}.
\end{thm}

As a consequence, we can see that if ageing is either very slow or very fast, then the behaviour of the original walk becomes irrelevant.
\begin{corollary}[Slow ageing and super ageing] Consider a one dimensional walk, and let $n\ge 2$. 
\begin{enumerate}
\item If $R^{\mathrm{pass}}= 1$ (``slow ageing''), then $\E_0 \widetilde {T}_n<\infty$ holds whatever the original walk is. 
\item If $R^{\mathrm{pass}}=0$ (``super-ageing'') then $\E_0 \widetilde {T}_n=\infty$ holds whatever the original walk is.
\end{enumerate}
\end{corollary}

\section{The spatial spread of the process}\label{speed.section}
Let $R_t$ denote the number of distinct edges crossed by $X^{\mathrm{imp}}$ up to actual time $t>0$.
If $G=\mathbb Z$, then $R_t=\max_{0\le s\le t}X^{\mathrm{imp}}_s-\min_{0\le s\le t}X^{\mathrm{imp}}_s$.

Assuming that the walk is strongly impatient with $\sum_k s_k=S$, clearly
$$
\liminf_{t\to\infty}\frac{R_t}{t}\ge \frac{1}{S}.
$$

\begin{problem}[Strongly impatient recurrent walk]
Assume that $X^{\mathrm{imp}}$ is strongly impatient and the classical random walk $X$ on $G$ is recurrent. Is it true for $X^{\mathrm{imp}}$ that
$$
 \lim_{t\to\infty}\frac{R_t}{t}=\frac{1}{S}?
$$
\end{problem}

\begin{problem}[Strongly impatient transient walk]
Assume that $X^{\mathrm{imp}}$ is strongly impatient and the classical random walk $X$ on $G$ is transient. Is it true for $X^{\mathrm{imp}}$ that
$$
 \lim_{t\to\infty}\frac{R_t}{t}=(0,\infty)?
$$ 
If so, what is the limit?
\end{problem}

\begin{problem}[Weakly impatient walk]
Assume that $X^{\mathrm{imp}}$ is weakly impatient. What is the asymptotic behaviour of $R_t$ as $t\to\infty$?
\end{problem}
\section{Comparison with classical ArcSine Law} 
Can one prove a generalized ArcSine Law? Recall that one way of formulating the classical ArcSine Law is that  the proportion of time spent on the right (left) by the walker has a limiting distribution. More precisely,  for $0<x<1$,
$$
\P(\text{the fraction of time units spent on the positive axis up to}\ n<x) \to \frac{2}{\pi} \arcsin \sqrt{x}
$$ 
as $n \to\infty$. (Another formulation is that if $k(n)$ denotes the  last return time to the origin up to~$2n$, then $k(n)/n$, has a limiting distribution.)
    
Now, in our case, the left-hand side depends on the passage times. Intuitively, if the random walk with the given passage times is impatient, then {\it the limiting distribution of the proportions is more balanced than in the classical case}. 

In fact, if the excursion time (between returns to the origin) has finite expectation (e.g. when~$M$ has finite expectation and passage times are strongly impatient, or the cases  computed in Section~\ref{general.formula}), then by the Renewal Theorem, the limit is completely balanced ($=1/2$)! 
\begin{problem}[Modified ArcSine Law]
What happens when~$X$ is simple random walk and the excursion time has infinite expectation? How will the passage times modify the ArcSine Law, making the limit more balanced?
\end{problem}
The following theorem may be considered as an initial step in this direction; it indicates that for strong enough impatience, a behaviour much more balanced than for the classical ArcSine Law is exhibited. 
\begin{thm}[Infinitely impatient RW]\label{inf.imp.thm}
On $\mathbb{Z}$, consider the ``infinitely impatient'' random walk, $X^{\mathrm{inf.imp}}$, that is, let $s_j=0, j\ge 1$. Let $R_n$ denote the time spent on the right axis up to time~$n\ge 0$, and~$L_n$ the time spent on the left axis up to time $n$. Then 
$$
\lim_{n\to\infty}\mathsf{Law}\left(\frac{R_{n}}{L_{n}+R_{n}}\right)=\mathsf {Uniform}([0,1]).
$$
\end{thm}
\begin{proof}
Let $\mathcal{R}^n_{l,r}$ denote the event that the range of the walk up to actual time $n\ge 2$ is $[l,r]$ with $l\le 0, r\ge 0$. Just like in Section~\ref{general.formula}, it is easy to see that
$$
\P (X^{\mathrm{inf.imp}}\ \text{reaches}\ l-1\ \text{before}\ r+1\mid X_n^{\mathrm{inf.imp}}=r,\ \mathcal{R}^n_{l,r})=\frac{1}{r-l+2}=\frac{1}{\text{size of range}-2}.
$$
Of course, for $X^{\mathrm{inf.imp}}$, the size of the range increases by one at each unit (actual) time.

Therefore, identifying right and left with ``heads'' and ``tails'',  $R_n$ can be identified with the number of heads in the following experiment:
We first toss a fair coin. 
Then we turn it over with probability $1/3$, and with probability $2/3$ we do nothing. Next we turn it over with probability~$1/4$, etc.
Finally, in the $n$-th step we turn it over with probability $1/(n-1)$. 

Using this equivalent formulation, the claim follows from Theorem 1 in \cite{coin}, where a more general ``coin turning'' model is investigated. 
\end{proof}
\begin{rema}
L.~Breiman~\cite{Breiman} proved a generalization of the ArcSine Law in the 1960's, and this was picked up by D. Mason et al~\cite{weighted.sums} recently. The point is that one can have a nice limit even if the law of the excursion is different from the classical one for simple random walk. (In~\cite{weighted.sums},  take~$X_i$ to be a random sign and $Y_i$ the excursion time for the $i$th excursion.)
\end{rema}

\section{Space-dependent impatience}
Here we modify the definition of impatience given in Section~\ref{Intro}. Suppose now that the passage times for an edge $e$ do not depend on the number of times the edge has been crossed, but rather on {\em the location} of this edge in the graph $G$;
thus $s_0(e)=s_1(e)=\dots =s(e)$, following our definition of the passage times. As before, fix some specific vertex $v_0$ of the graph and call it {\em the origin}. For a vertex $v$ in $G$ let the distance from $v_0$ to $v$, denoted by $\Vert v\Vert\in\{0,1,2,\dots\}$, be the number of edges on the shortest path connecting $v_0$ and $v$, and for an edge $e=(v_1,v_2)$ let $\Vert e \Vert=\min\{\Vert v_1\Vert,\Vert v_2\Vert\}$. 

Let $X$ be a unit time random walk on $G$, i.e.\ a Markov chain whose transitions are restricted to the edges of $G$. It is then conceivable that while $X_n$ is null-recurrent, $\Ximp$ may still be positive recurrent, provided $s(e)\to 0$ sufficiently quickly as $\Vert e\Vert\to \infty$.

For any two vertices $v$ and $u$ let us define 
\begin{align*}
p(u,v)&=\P^u(X\text{ hits $v$ before ever returning to $u$})\\
&=\P^u(\exists n\ge 1: \ X_1\notin\{u,v\},\dots,
X_{n-2}\notin\{u,v\}, X_{n-1}\notin\{u,v\}, X_n=v)\\
&=\P^u(\sigma_v<\sigma_u)
\end{align*}
where $\sigma_u:=\min\{k\ge 1\mid X_k=u\}$  for  $u$ ($\sigma_v$ is defined analogously).
\begin{assumption}\label{Stas.assum}
We assume the following  about the random walk $X$.

\begin{itemize}
  \item (Uniform ellipticity)   There is a universal constant $\eps>0$ such that for any edge $e=(v_1,v_2)$ in the graph $\P(X_{n+1}=v_2\| X_n=v_1)\ge \eps$, $\P(X_{n+1}=v_1\| X_n=v_2)\ge \eps$.
  \item (Return symmetry)  
  There is a universal constant $\rho\ge 1$ such that for any $v$ in the graph
  $$
  \rho^{-1} \, p(v,v_0)\le p(v_0,v)\le \rho\, p(v,v_0).
  $$
\end{itemize}
\end{assumption}

\begin{rema}\label{rem11}
Observe that:
\\
(a) uniform ellipticity implies that the graph is of uniformly bounded degree, i.e.\ there is $D\ge 1$ such that each vertex of the graph has at most $D$ edges coming out of it, and that there are no oriented edges on $G$;
\\
(b)
return symmetry implies that the underlying random walk cannot be positive recurrent.
\end{rema}
Note that part (b) of the above remark follows from the following observation. For each $v\ne v_0$, the walk starting at~$v_0$ has the probability $p(v_0,v)$ of hitting $v$ before returning to the origin~$v_0$. However, after reaching $v$, the walk makes a geometric number of returns to $v$ itself before coming back to~$v_0$. The expected number of such returns, including the very first visit, is $\frac 1{p(v,v_0)}$. Thus, the total expected number of vertices visited by the walk (with multiplicity) equals $\sum_{v\in G} \frac {p(v_0,v)}{p(v,v_0)}$ which is infinite since each term is at least $\rho^{-1}$. Hence the expected number of steps the walk makes before returning to $v_0$ is also infinite.

Later we will show that SRW both on $\Z$ and $\Z^2$ satisfy the above assumptions.

\begin{thm}\label{thm2ass}
Let $X$ be a null-recurrent random walk on $G$, satisfying Assumption~\ref{Stas.assum},  with $X_0=v_0$. Then $\Ximp$ is positive recurrent if and only if 
$$
\sum_{e\in E(G)} s(e)<\infty.
$$
\end{thm}
\begin{proof}
First of all, it is clear that
$$
\E\tilde\tau=\E \sum_{e\in E(G)} \xi(e) s(e)
= \sum_{e\in E(G)} s(e) \cdot \E \xi(e) 
$$
where $\xi(e)$ is the number of times  edge $e$ is crossed (in either direction) prior to $\sigma_{v_0}$; formally, for $e=(u_1,u_2)$,
$$
\xi(e)=\sum_{n=0}^{\sigma_{v_0}-1} \left[\mathbf{1}_{\{X_n=u_1,X_{n+1}=u_2\}}
+\mathbf{1}_{\{X_n=u_2,X_{n+1}=u_1\}}\right].
$$ 
Consequently, to establish the statement of the theorem it suffices to show that $\E \xi(e)$ are bounded above and below by some positive constants not depending on $e$. 

To this end, observe first that $p(v_0,v)\to 0$  as $\Vert v\Vert \to \infty$, since $X$ is recurrent. Indeed, for~$n\ge 1$, one has
$$
p(v_0,v)\le \P^0(\sigma_v\le n)+\P^0(\sigma_{v_{0}}\ge n),
$$
where $\P^0(\cdot):=\P(\cdot\mid X_0=v_0)$. Now $\P^0(\sigma_{v_{0}}\ge n)$ tends to zero as $n\to\infty$ and $\P^0(\sigma_v\le n)=0$ whenever $\Vert v\Vert >n.$

Next, note  that the number of vertices at distance at most $k$ from $v_0$ is bounded above by $\sum_{i=1}^k d(d-1)^{i-1}<\infty$, so we can safely assume from now on that $p(v_0,v)$ is small.

Let $u\in G$  and with a slight abuse of notations let $\xi(u)=\sum_{n=1}^{\sigma_{v_0}} 1_{X_n=u}$ be the number of times $u$ is visited before $\sigma_{v_0}$. Let $e=(u_1,u_2)\in E(G)$. Since, after each visit of $u_i$, $i=1,2$, the walk crosses $e$ with probability at least $\eps$, and to cross $e$ it {\em has to} visit one of the endpoints, we have
$$
\eps\cdot \E\left[\xi(u_1)+\xi(u_2)\right]  \le \E \xi(e)\le \E\left[\xi(u_1)+\xi(u_2)\right].
$$
(In fact, the right inequality holds even without expectation signs!) Therefore, if we show that the ~$\E\xi(u)$ are bounded above and below, uniformly over $u\in G$, then we are done.

However, once vertex $u$ is reached (and this happens with probability $p(v_0,u)$) the number of returns to it before hitting $v_0$ is geometric with success probability $p(u,v_0)$ and mean  $\frac{1}{p(u,v_0)}$, and consequently,
$$
\E \xi(u)=  [1-p(v_0,u)]\cdot 0+p(v_0,u)\cdot \frac{1}{p(u,v_0)},
$$
which belongs to the interval $[\rho^{-1},\rho]$ for all $u$, because of the second part of Assumption~\ref{Stas.assum} (see also the argument after Remark~\ref{rem11}). The theorem is thus proven.
\end{proof}

Recall that the graph $G$ is called {\it transitive} if,  viewed from any vertex $v$ in $G$, it is isomorphic to $G$ viewed from $v_0$.
\begin{prop}\label{prop2ass}
Let  $X$ be a recurrent simple\footnote{I.e. $X$ jumps to each neighbour with the same probability.} random walk on the transitive graph $G$. Then Assumption \ref{Stas.assum} is satisfied.
\end{prop}
\begin{proof}
The uniform ellipticity assumption is trivially satisfied since $X$ is a symmetric random walk, and each vertex is incident to the same number of edges, because of transitivity of $G$.
Now we have $p(v_0,v)=p(v,v_0)$ by transitivity again, and thus one can set $\rho=1$.
\end{proof}

\begin{corollary}
The symmetric random walks on $\Z^1$ and on $\Z^2$ satisfy Assumption~\ref{Stas.assum}.  Assuming $v_0={\bf 0}$, if $s(e)\sim \Vert e\Vert ^{-\a}$ for some $\a>0$, then $\Ximp$ is positive recurrent if and only if
\begin{itemize}
\item $\a>1$, in case of $\Z^1$; 
\item $\a>2$, in case of~$\Z^2$.
\end{itemize}
\end{corollary}
\begin{proof}
Both $G=\Z^1$ and $G=\Z^2$ are transitive, and the symmetric random walks on them are recurrent. Hence, by Proposition~\ref{prop2ass}, Assumption~\ref{Stas.assum} is fulfilled. Consequently, the positive recurrence of~$\Ximp$ is equivalent to the finiteness of
$$
\sum_{e\in E(G)} s(e)\sim
\begin{cases}
  \sum_{k=1} \frac{2}{k^\a} \sim \sum_{k} k^{-\a}
&\text{ on }\Z^1;\\
 \sum_{k=1} \frac{4k}{k^\a} \sim \sum_{k} k^{-(\a-1)} 
&\text{ on }\Z^2,
\end{cases}
$$
yielding the required result.
\end{proof}

%

\end{document}